\documentclass[11pt]{amsart}
\usepackage{amsmath,amsfonts,amsthm,amscd,graphicx,amssymb,verbatim}
\usepackage{enumerate}
\usepackage{tikz}
\usetikzlibrary{calc}
\usetikzlibrary{arrows}
\usepackage{soul}

\usepackage{color}

\newtheorem{theorem}{Theorem}[section]
\newtheorem{lemma}[theorem]{Lemma}

\theoremstyle{definition}
\newtheorem{definition}[theorem]{Definition}
\newtheorem{example}[theorem]{Example}

\newtheorem{cor}[theorem]{Corollary}
\newtheorem{prop}[theorem]{Proposition}

\theoremstyle{remark}

\numberwithin{equation}{section}

\newcommand{\ZS}{\mathcal{S}}
\newcommand{\ZE}{\mathcal{E}l}

\begin{document}

\title[Continuous Sensitivity and Reversibility]{\bf Continuous Sensitivity and Reversibility}
\author{Asl\i \ G\"{u}\c{c}l\"{u}kan \.{I}lhan and \"{O}zg\"{u}n \"{U}nl\"{u}}

\address{Asli G\" u\c cl\" ukan \.{I}lhan\\
Department of Mathematics \\
Dokuz Eyl\"{u}l University\\
Buca, \.Izmir, Turkey.}

\email{asli.ilhan@deu.edu.tr}

\address{\"{O}zg\"{u}n \"{U}nl\"{u} \\
Department of Mathematics \\
Bilkent University \\
Bilkent, Ankara, Turkey.}

\email{unluo@fen.bilkent.edu.tr}

\thanks{The second author is partially supported by T\"{U}BA-GEBİP/2013-22.}


\begin{abstract}

Let $n$ be a positive integer and $f$ a differentiable function from a convex subset $C$ of the Euclidean space $\mathbb{R}^n$ to a smooth manifold. We define an invariant of $f$ via counting certain threshold functions associated to $f$. We  call this invariant the continuous sensitivity of $f$ and denote it by $\mathrm{cs}_{C}(f)$. This invariant is a real number between $0$ and $n$ and measures how sensitive $f$ is to change in its input variables. For example, if $f$ is a constant function then $\mathrm{cs}_{C}(f)=0$. On the other extreme, if $\mathrm{cs}_{C}(f)=n$ then $f$ is one-to-one on $C$.  This last statement is important for reversibility problems. To say that a function is reversible one can write an explicit inverse of the function. However, this is not always easy. Even a multilinear function can have a complicated inverse function. Here we give tools to compute continuous sensitivity which makes it possible to answer reversibility problems without finding explicit inverse functions.
\end{abstract}

\maketitle

\tableofcontents

\section{Introduction}

In this paper we define (Definition \ref{definition:contsens}) and study the continuous sensitivity of a differentiable function  from a convex subset of an Euclidean space to a smooth manifold. The main reason behind the definition of continuous sensitivity is that it is a useful  invariant which helps us to answer (Corollary \ref{cor:revFirst}) reversibility problems without finding an explicit inverse of the function. More precisely we show that such a function is one-to-one when its continuous sensitivity is equal to the dimension of the Euclidean space in which its domain lives. We also give several tools (Theorem \ref{thm:counting-intersctions}, Proposition \ref{prop:counting-set}) for computing continuous sensitivity of certain functions using only a finite amount of information. The rest of this introduction is to motivate the reader about continuous sensitivity and reversibility problems by giving directions for applications of these concepts in areas like algebraic topology and computer science.

Many topological spaces of interest can be constructed by using simplices as building blocks like simplicial complexes, realizations of categories, and manifolds with triangulations. Hence continuous sensitivity can be used to detect homeomorphisms between such topological spaces. For example, topological realization of a category enriched over simplicial sets is constructed by gluing products of simplicies like $ \Delta ^{n_1-1}\times \Delta ^{n_2-1}\times \dots \times \Delta ^{n_k-1}$ and the topological realization of a category enriched over sets is constructed by gluing simplicies like $\Delta^{m-1}$ to each other. Therefore to compare the topological realization of a category enriched over simplicial sets with the topological realization of a category enriched over sets, one has to study multilinear functions in the following form
$$\overline{\phi }: \Delta ^{n_1-1}\times \Delta ^{n_2-1}\times \dots \times \Delta ^{n_k-1}\rightarrow \Delta^{m-1}. $$
Here the domain of $\overline{\phi }$ is a convex subspace of the Euclidean space $\mathbb{R}^n$ where $n=n_1+n_2+\dots +n_k$ and the codomain of $\overline{\phi }$ is a subspace of $\mathbb{R}^m$. Moreover, these multilinear functions are differentiable. Hence, one can determine the reversibility of $\overline{\phi }$  by computing its continuous sensitivity. 

Let $\mathcal{C}$ be a category which contains a pair of composable non-identity morphisms. In the geometric realization of $\mathcal{C}$, there is an associated $2$-simplex $\Delta ^2$ for every such pair of morphisms. Let $\mathcal{D}$ be a category enriched over simplicial sets with morphisms $f$, $g$, and $h$ such that $h$ and $g\circ f$ are homotopic in  $\mathcal{D}$.  Then  in the realization of  $\mathcal{D}$,  there exist associated $2$-simplexes and simplicial sets of the form $\Delta ^1 \times \Delta ^1 $ glued to each other as discussed in Example \ref{ex:cat}. In this example, we use the techniques discussed above to  write  a homeomorphism between these parts of realizations of $\mathcal{C}$ and $\mathcal{D}$.

Other applications can be found in computer science. Multi-valued logic is a propositional calculus where the logical operations have input variables and an output variable with possibly different sets of truth values. More precisely a multi-valued logic gate $\phi $ is a function from a product of sets $T_1\times T_2\times \dots \times T_k$ to a set $T$ where the set $T_i$ is the set of possible truth values of the $i^{\text{th}}$ input variable and the set $T$ is the set of possible truth values of the output. In nature, most of the time logic gates communicate with each other using analogue signals. 
Hence it is natural to consider $T=\{v_1,v_2,\dots v_n \}$ where $v_1$, $v_2$, $\dots$, $v_n$ are vectors in $\mathbb{R}^m$ for some $m\geq 1$. In this case, the Fourier-expansion of the multi-valued logic gate is the multilinear function
$$\overline{\phi }: \Delta ^{n_1-1}\times \Delta ^{n_2-1}\times \dots \times \Delta ^{n_k-1}\rightarrow \mathbb{R}^m $$
given by
$$\overline{\phi }((t_{1,j})_{j\in T_1},(t_{2,j})_{j\in T_2},\dots,(t_{k,j})_{j\in T_k})
=\sum _{j\in T} \left( \sum_{\phi(j_1,\dots ,j_k)=v_j} \,\, \prod _{s=1}^{k}t_{s,j_s} \,\, \right) v_j $$
where $n_i$ is the number of elements in $T_i$. Notice that the output of this gate lives in the convex hull spanned by the vectors $T=\{v_1,\dots v_n\}$. For the Fourier series expansions of Boolean functions, see \cite{Bruck}, \cite{GotsmanLinial}, \cite{ODonnell}.

For example, our eyes can only observe red, green, blue light and lack of light namely black. Hence the convex hull created by these can be considered as a 3-simplex $\Delta ^3$ where in general we have
$$\Delta ^n=\{(t_0, t_1, \dots , t_n)\,|\,\sum _{i=0}^n t_i=1\}.$$
Here, if ``black" corresponds to the point $(1,0,0,0)$. Then for $t_0 = 0$ we obtain a color triangle which could be considered as $\Delta ^2$.
One could take the truth values for color as the vertices in a barycentric subdivision of this triangle in particular if ``red" corresponds to $(1,0,0)$ and ``blue" corresponds to $(0,0,1)$ then purple will correspond to $(1/2,0,1/2)$.  In Example \ref{ex:color}, we discuss a logical gate that sends $\Delta^1 \times \Delta^1 \times \Delta^1$ to $\Delta ^3$ and show the two convex subsets are homeomorphic.

We also develop tools that can be used to analyse the continuous sensitivity of a multi-valued logic gate using experimental data about the gate, which makes it a computable and useful invariant to compare logic gates. As an application we show that  continuous sensitivity provides a lower bound for sensitivity of a boolean function (see \cite{Ambainisetal}, \cite{KenyonKutin}, \cite{Nisan}, \cite{Sobol}, \cite{Wang}) considered as a multi-valued logic gate. Another important issue to consider about multi-valued logic gates is reversibility (see \cite{Rabadi}, \cite{Cheng}, \cite{Kutrib}, \cite{Peres}, \cite{Toffoli}, \cite{rev}). We show that a multi-valued logic gate is one-to-one when the continuous sensitivity is equal to a certain number (see Corollary \ref{cor:rev}).

\section{Definition of Continuous Sensitivity}

For $x$ in $\mathbb{R}$, we have
\begin{eqnarray*} \mathrm{sgn}(x)=\left\{
	\begin{array}{ll}
		+1, & \hbox{if $x>0$;} \\
		0, & \hbox{if $x=0$;} \\
		-1, & \hbox{if $x<0$.}
	\end{array}
	\right.
\end{eqnarray*}

Given a function $g:\mathbb{R}^n\rightarrow \mathbb{R}$,  the composition $\mathrm{sgn}\circ g$ is a threshold function.  In this paper we discuss such functions using the following definition.

\begin{definition} Let $C$ be a nonempty subset of $\mathbb{R}^n$. We define the sign of $g:\mathbb{R}^n\rightarrow \mathbb{R}$ over $C$ as follows:
\begin{eqnarray*} \mathrm{sign}_C(g)=\left\{
	\begin{array}{ll}
		+1, & \hbox{if $\mathrm{sgn}( g (C))=\{+1\}$;} \\
		0, & \hbox{if $\mathrm{sgn}( g (C))=\{0\}$;} \\
		-1, & \hbox{if $\mathrm{sgn}( g (C))=\{-1\}$;} \\
        u, & \hbox{otherwise.}
	\end{array}
	\right.
\end{eqnarray*}
\end{definition}
If $g$ and $C$ are as above and the function $g(x_1,x_2,\dots , x_n)$ is differentiable on $C$, then we define the total sign of $g$ over $C$ as follows:
$$\mathrm{Sign}_C(g)=\left\langle  \mathrm{sign}_C\left(\frac{\partial g}{\partial x_1}\right),\mathrm{sign}_C\left(\frac{\partial g}{\partial x_2}\right),\dots, \mathrm{sign}_C\left(\frac{\partial g}{\partial x_n}\right)\right\rangle. $$

Let $\ZS_n$ be the set of all non-zero $n$-tuples $\langle s_1, s_2, \dots , s_n \rangle $ in $\{-1,0,1\}^n$ whose first non-zero term is $1$. We say that a tuple $t=\langle t_1,t_2, \dots,t_n \rangle$ in $\{1,0,-1,u\}^n$ eliminates a tuple $s=\langle s_1,s_2,\dots,s_n \rangle$ in $\ZS _n$ if the following conditions hold
\begin{itemize}
	\item[i)] $t_i \neq 0$ and $s_i \neq 0$ for some $i$,
	\item[ii)] there exists $k \in \{ +1 ,-1 \}$ such that $t_i=ks_i$, for all $i$ with $s_i \neq 0$ and $t_i \neq 0,$
	\item[iii)] $s_i=0$ when $t_i=u$.
\end{itemize}
For $X \subseteq \ZS_n$, we denote the set of elements of $\ZS_n$ eliminated by an element of $X$ by $\ZE (X)$.

Let $C$ be a convex subset of $\mathbb{R}^n$, $M$ be a smooth manifold, and $f:\mathbb{R}^n \rightarrow M$ be a differentiable function.
Now we  define a set assosiated to $f$ as follows:
$$\mathrm{Sens}_C(f)=\left\{\, v\in \ZS _{n}\,\left| \,
\begin{array}{c}
\text{There exists } \pi: M \rightarrow \mathbb{R} \text{ a differentiable  function} \\
\text{such that }\mathrm{Sign}_C( \pi\circ f)  \text{ eliminates }v
\end{array}\, \right. \right\}.$$
In other words
$$\mathrm{Sens}_C(f)=\ZE \left\{\,\,\, \mathrm{Sign}_C( \pi\circ f) \,\,\,\left| \,\,\,
 \pi: M \rightarrow \mathbb{R} \text{ is a differentiable  function }\,\right. \right\}.$$
Note that the larger the set $\mathrm{Sens}_C(f)$ is the more sensitive the function $f$ is to its input variables. Hence we make the following definition.
\begin{definition}\label{definition:contsens} Let $C$ be a convex subset of $\mathbb{R}^n$, $M$ be a smooth manifold, and $f:\mathbb{R}^n \rightarrow M$ be a differentiable function. Then we define continuous sensitivity of $f$ on $C$ as follows:
$$\mathrm{cs}_{C}(f)=\log _{3}\left( 3^{n}-2\left| \ZE(\ZS _{n} - \mathrm{Sens}_{C}(f))\right|\right).$$
\end{definition}
Now we specialize this definition for multi-valued logic gates. Let $k$ be a natural number. For $i$ in $\{1,2,\dots ,k\}$,  let $n_i$ be a natural number and
$$T_i=\{w(i,0), w(i,1), \dots, w(i,n_i-1)\}$$
be the set of possible truth values that we could put in for the $i^{\text{th}}$ variable. Let
$$T=\{v_1,v_2,\dots v_n\}$$
be a set of vectors in $\mathbb{R}^m$. We will consider the elements in $T$ as truth values of the output. A multi-valued logic gate $\phi $ is a function from $T_1\times T_2\times \dots  \times  T_k$ to $T$. Given a multi-valued logic gate $\phi: T_1\times T_2\times \dots  \times  T_k \rightarrow T$, we define the Fourier series expansion of $\phi$ as the function
$$\overline{\phi }: \Delta ^{n_1-1}\times \Delta ^{n_2-1}\times \dots \times \Delta ^{n_k-1}\rightarrow \mathbb{R}^m$$
given by
$$\overline{\phi }((t_{1,j})_{j=0}^{n_1-1},(t_{2,j})_{j=0}^{n_2-1},\dots,(t_{k,j})_{j=0}^{n_k-1})
=\sum _{j=0}^{n}
\left( \sum_{\phi(w(1,j_1),\dots ,w(k,j_k))=v_j} \,\,\prod _{s=1}^{k}t_{s,j_s} \,\, \right) v_j .$$

Take an element $z$ in $T_1\times T_2\times \dots  \times  T_k$. We can write $z$ in the following form
$$z=(w(1,j(z,1)),\dots ,w(k,j(z,k)))$$
where $0\leq j(z,i)\leq n_i$. Let $N(\overline{\phi })=n_1+n_2+\dots +n_k-k.$ We define $C(\overline{\phi },z)$ a convex subset of $\mathbb{R}^{N(\overline{\phi })}$ as follows:
$$C(\overline{\phi },z)=\Delta _{j(z,1)}^{n_1-1}\times \Delta _{j(z,2)}^{n_2-1}\times \dots \times \Delta _{j(z,k)}^{n_k-1}$$
where  $\Delta _{j}^{m}=\{(t_0, \dots, \widehat{t_j} , \dots, t_m)\,|\,(t_0,t_1, \dots , t_m)\in \Delta ^m \text{ and } t_j\neq  0 \}$ for natural numbers $j\leq m$.
Now we define a differentiable function $f(\overline{\phi })$ from $\mathbb{R}^{N(\overline{\phi })}$ to $\mathbb{R}^m$ by
$$f(\overline{\phi })(t_{1,1},\dots ,\widehat{t_{1,j(z,1)}},\dots , \widehat{t_{k,j(z,k)}},\dots , t_{k,n_k-1})= \overline{\phi }((t_{1,j})_{j=0}^{n_1-1},(t_{2,j})_{j=0}^{n_2-1},\dots,(t_{k,j})_{j=0}^{n_k-1})$$
where the right-hand side is considered to be defined everywhere by seeing each component of the right-hand side as a multilinear polynomial and taking
$$t_{i,j(z,i)}=
1-\underset{\scriptsize{\begin{array}{c}
1\leq j \leq n_i-1\\
j\neq j(z,i)
\end{array}}}{\sum} t_{i,j}.$$
\begin{definition} The continuous sensitivity of $\overline{\phi }$ at $z$ is defined as follows:
$$\mathrm{cs}(\overline{\phi },z)= \mathrm{cs}_{C(\overline{\phi },z)}(f(\overline{\phi })) .$$
The continuous sensitivity of $\overline{\phi }$ is defined to be the maximum among them:
$$\mathrm{cs}(\overline{\phi })=\max \{\mathrm{cs}(\overline{\phi },z)\,|\,z\in T_1\times T_2\times \dots \times T_k\}.$$
\end{definition}
The larger this number is the more sensitive the multi-valued logic gate is to its input variables. We will explain this in the next sections.

\section{Reversibility}

The main theorem of this section is the following theorem.
\begin{theorem}\label{thm:rev} Let $C$ be a convex subset of $\mathbb{R}^n$, $M$ be a smooth manifold, and $f:\mathbb{R}^n \rightarrow M$ be a differentiable function. If  $\mathrm{Sens}_C(f)= \ZS _n$ then $f$ is one-to-one on $C$.
\end{theorem}
\begin{proof} Assume that there are two distinct points $x,y$ in $C$ such that $f(x)=f(y)$. Let $v=y-x=<v_1,\dots,v_n>$. Since $x \neq y$ there exists $i$ such that $v_i \neq 0$. Hence $\mathrm{sgn}(v)=(\mathrm{sgn}(v_1),\dots ,\mathrm{sgn}(v_n))$ is in $\ZS^*_n$. Therefore there exists a differentiable function $\pi:  M \rightarrow \mathbb{R}$ such that $\mathrm{Sign}_C(\pi \circ f)$ eliminates $\mathrm{sgn}(v)$. Let $\gamma :[0,1]\to \mathbb{R}^n$ be the linear parametrization of line segment from $x$ to $y$. Since for all $j$ we have
	$$v_j \left( \frac{\partial
	}{\partial x_j}(\pi\circ f)|_{\gamma(t)}\right)\geq 0$$
	and the equality doesn't hold for at least one $j$, we get a contradiction as follows:
	$$0=(\pi\circ f)(y)-(\pi\circ f)(x)= \int
	_{0}^{1}\frac{d}{dt}\left((\pi\circ f\circ \gamma)
	(t)\right)\, dt=$$
	$$=\int _{0}^{1}\sum _{j=1}^n v_j \left( \frac{\partial
	}{\partial x_j}(\pi\circ f)|_{\gamma(t)}\right)\,
	dt>0$$
	
\end{proof}
As consequence of this result we obtain the following result.
\begin{cor}\label{cor:revFirst}
Let $C$ be a convex subset of $\mathbb{R}^n$, $M$ be a smooth manifold, and $f:\mathbb{R}^n \rightarrow M$ be a differentiable function. If $\mathrm{cs}_{C}(f)=n$ then $f$ is a one-to-one function on $C$.
\end{cor}
\begin{proof} If $\mathrm{cs}_{C}(f)=n$ than $\ZS _{n}=\mathrm{Sens}_{C}(f)$. Therefore $f$ is one-to-one on  $C$ by the above theorem.
\end{proof}
For multivalued logic gates we have the following analogous result.
\begin{cor}\label{cor:rev} Let $\overline{\phi }$ be the Fourier series expansion of a multi-valued logic gate. If $\mathrm{cs}(\overline{\phi })=N(\overline{\phi })$ then $\overline{\phi }$ is a one-to-one function on the interior of $\Delta ^{n_1-1}\times \Delta ^{n_2-1}\times \dots \times \Delta ^{n_k-1}$.
\end{cor}
\begin{proof} First notice that $\mathrm{cs}(\overline{\phi })=N(\overline{\phi })$ means
$$\mathrm{cs}(\overline{\phi },z)=N(\overline{\phi })$$
for some $z$. By the above corollary, this means that $f(\overline{\phi })$ is one-to-one on  $C(\overline{\phi },z)$. So $\overline{\phi }$ is one-to-one on the interior of $\Delta ^{n_1-1}\times \Delta ^{n_2-1}\times \dots \times \Delta ^{n_k-1}$.
\end{proof}

Due to the above results one can see that it is important to study the minimal elements of the following poset
$$E(\ZS_n)=\{X \subseteq \ZS_n| \ \ZE(X)= \ZS_n\}$$
under inclusion.

\begin{example}\label{ex:base}When $n=2$, $\ZS_2=\{(0,1),(1,0),(1,1),(1,-1)\}$ and
\begin{eqnarray*} \ZE(\{(0,1)\})&=&\{(0,1),(1,1),(1,-1)\},\\
                 \ZE(\{(1,0)\})&=&\{(1,0),(1,1),(1,-1)\},\\
                 \ZE(\{(1,1)\})&=&\{(0,1),(1,0),(1,1)\},\\
                 \ZE(\{(1,-1)\})&=&\{(0,1),(1,0),(1,-1)\}.\\
\end{eqnarray*} Therefore $E(\ZS_n)$ is the set of subsets of $\ZS_2$ of size greater than equal to $2$ and the minimal elements of $E(\ZS_2)$ are the subsets of $\ZS_2$ of size $2$.
\end{example}

\begin{example}\label{ex:cat} Let $f_1,f_2:\Delta^1 \times \Delta^1 \rightarrow \Delta^2$ and $f_i: \Delta^2 \rightarrow \Delta^2$ for $3 \leq i \leq 6$ be the continuous functions defined by

\begin{eqnarray*}
f_1((t_0,t_1),(s_0,s_1))&=&(\frac{t_0s_1}{2}+\frac{t_1s_1}{3}, \frac{t_1s_1}{3}, s_0+\frac{t_0s_1}{2}+\frac{t_1s_1}{3})\\
f_2((t_0,t_1),(s_0,s_1))&=&(s_0+\frac{t_0s_1}{2}+\frac{t_1s_1}{3},\frac{t_1s_1}{3}, \frac{t_0s_1}{2}+\frac{t_1s_1}{3})\\
f_3(t_0,t_1,t_2)&=&(t_0+\frac{t_1}{2}+\frac{t_2}{3}, \frac{t_1}{2}+\frac{t_2}{3} ,\frac{t_2}{3})\\
f_4(t_0,t_1,t_2)&=&(\frac{t_1}{2}+\frac{t_2}{3}, t_0+\frac{t_1}{2}+\frac{t_2}{3},\frac{t_2}{3})\\
f_5(t_0,t_1,t_2)&=&(\frac{t_2}{3},\frac{t_1}{2}+\frac{t_2}{3}, t_0+\frac{t_1}{2}+\frac{t_2}{3})\\
f_6(t_0,t_1,t_2)&=&(\frac{t_2}{3},\frac{t_1}{2}+\frac{t_2}{3}, t_0+\frac{t_1}{2}+\frac{t_2}{3})\\
\end{eqnarray*}
The image of $f_i$ is the region denoted by $i$ in the following picture:
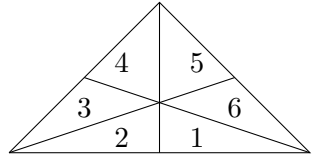
\begin{center}
\begin{tikzpicture}
\draw (0,0) -- (2,2) -- (4,0)-- (0,0);
\draw (0,0)--(3,1);
\draw (1,1)--(4,0);
\draw (2,0) --(2,2);
\draw (2.5,0.2) node{$1$};
\draw (1.5,0.2) node{$2$};
\draw (1,0.6) node{$3$};
\draw (1.5,1.2) node{$4$};
\draw (2.5,1.2) node{$5$};
\draw (3,0.6) node{$6$};
\end{tikzpicture}
\end{center}
Let $f^j_i=\pi_j \circ f_i$ where $\pi_j: \mathbb{R}^3 \rightarrow \mathbb{R}$ is the projection to the $j$-th coordinate. Then we have

\begin{eqnarray*}
\frac{\partial f_1^1}{\partial t_1}&=&\frac{-s_1}{6}<0 \ \mathrm{if} \ s_1\neq 0, \qquad \frac{\partial f_1^1}{\partial s_1}=\frac{t_0}{2}+\frac{t_1}{3}>0,\\
\frac{\partial }{\partial t_1}(f_1^2-f_1^3)&=&0, \qquad  \qquad \ \quad  \frac{\partial}{\partial s_1}(f_1^2-f_1^3)=1-\frac{t_0}{2}>0.\end{eqnarray*} Here  $\mathrm{Sign}_A(f_1^1)=(-1,1)$ and $\mathrm{Sign}_A(f_1^2-f_1^3)=(0,1)$ where $A$ is the set of all points $((t_0,t_1),(s_0,s_1)) \in |\Delta^1|\times |\Delta^1|$ with $ s_1\neq 0$. Therefore by above Theorem and Example, $\left.f_1\right|_A$ is injective .

For $f_3$, we have
\begin{eqnarray*}
\frac{\partial f^1_3}{\partial t_1}&=&\frac{1}{6}, \qquad   \frac{\partial f_3^1}{\partial s_1}=\frac{-5}{6}, \qquad 
\frac{\partial f^3_3}{\partial t_1}=\frac{-1}{3},   \qquad \frac{\partial f_3^3}{\partial s_1}=\frac{-1}{3}.\\
\end{eqnarray*} Since $\mathrm{Sign}_{\Delta^2}(f_3^1)=(1,-1)$ and $\mathrm{Sign}_{\Delta^2}(f_3^3)=(-1,-1)$, the map $f_3$ is injective . Similarly, one can show that the maps $\left.f_2\right|_A$, $f_4$, $f_5$ and $f_6$ are injective. Indeed, we can glue them to obtain a homeomorphism between the following spaces.
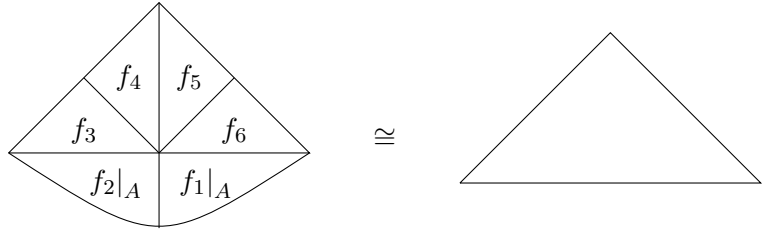
\begin{center}
\begin{tikzpicture}
\draw (0,0) -- (2,2) -- (4,0)-- (0,0);
\draw (1,1)--(2,0);
\draw (3,1)--(2,0);
\draw (2,0) --(2,2);
\draw (2,0) --(2,-1);
\draw (0,0) .. controls (2,-1.3) .. (4,0);
\draw (2.6,-0.4) node{$\left.f_1\right|_A$};
\draw (1.4,-0.4) node{$\left.f_2\right|_A$};
\draw (1,0.3) node{$f_3$};
\draw (1.6,1) node{$f_4$};
\draw (2.4,1) node{$f_5$};
\draw (3,0.3) node{$f_6$};
\draw (5,0.2) node{$\cong$};
\draw (6,-0.4) -- (8,1.6) -- (10,-0.4)-- (6,-0.4);
\end{tikzpicture}
\end{center}
Here the simplicial complex on the left is homeomorphic to the part of the realization of the category $\mathcal{D}$ discuss in the introduction.
\end{example}

\begin{prop} Let $v \in \mathbb{R}^n$ and $s$ be the sign vector of $v$. If either $s$ or $-s$ is in $\ZE(x)$ then $v \cdot x \neq 0.$
\end{prop}
Note that here either $s$ or $-s$ is an element of $\ZS_n$.

\begin{proof} Let $s=(s_1,\dots,s_n)$ be the sign vector of $v=(v_1,\dots,v_n)$. Without loss of generality suppose that $s=(s_1,\dots,s_n)$ is in $\ZE(x)$ for some $x=(x_1,\dots,x_n) \in \ZS_n$. Let $\{i_1,\dots,i_j\}$ be the set of all indices for which $s_i \neq 0$ and $x_i \neq 0$. Since $s \in \ZE(x)$, there exists $k\in \{\pm 1\}$ such that $s_{i_r}=kx_{i_r}$ for all $1\leq r \leq j$. By definition $s_i=0$ if and only if $v_i=0$. Therefore we have
$$v\cdot x=\sum_{r=1}^j v_{i_r}x_{i_r}=\sum_{r=1}^j|v_{i_r}|s_{i_r}x_{i_r}=k\sum_{r=1}^j|v_{i_r}|(x_{i_r})^2 \neq 0.$$
\end{proof}
Given a subset $X=\{X^1,\dots,X^m\}$ of $\ZS_n$, let $M_X$ be the $(m \times n)$-matrix whose $i$-th row is $X^i$. As an immediate consequence of the above proposition, we have the following results.
\begin{cor} Let $X=\{X^1,\dots,X^m\}\subseteq \ZS_n$. If the columns of $M_X$ is linearly dependent then $X \notin E(\ZS_n)$. In particular, if $X \in E(\ZS_n)$ has size $n$ then $X$ is linearly independent.
\end{cor}
\begin{proof} Let $a_1, \dots, a_n \in \mathbb{R}$ be such that $\overset{n}{\underset{i=1}{\sum}}a_ic_i=0$ where $c_i$ is the $i$-th column of $M_X$. Suppose also that the first non-zero term of $a=(a_1,\dots,a_n)$ is positive, that is, $a \in \ZS_n$. Since $a\cdot X^i=0$ for $1\leq i \leq m$, the sign vector $s$ of $a$ is in $\ZS_n - \ZE(X)$ by the above proposition and hence $X \notin E(\ZS_n)$.
\end{proof}
\begin{cor} If $X \in E(\ZS_n)$ then $|X| \geq n$.
\end{cor}
\begin{proof} If $|X|<n$ then we can choose a vector $v$ which is orthogonal to all the vectors in $X$.
\end{proof}

Note that the inequality in the above corollary is strict since $E=\{e_1,\dots,e_n\}$ is in $E(\ZS_n)$ where $e_k=(0,\dots,0,\underbrace{1}_{k},0,\dots,0)$. Since every element of $E(\ZS_n)$ of size $n$ is minimal, $E$ is indeed a minimal element of $E(\ZS_n)$.

Let $\ZS_n^0$ be the set of all elements of $\ZS_n$ with non-zero coordinates. These elements are also the ones which eliminates the smallest number of elements of $\ZS_n$.
\begin{lemma} The subset $\ZS_n^0$ is a minimal element of $E(\ZS_n)$.
\end{lemma}
\begin{proof} Since $s=(s_1,\dots,s_n)\in \ZS_n$ is eliminated by all $t=(t_1,\dots,t_n)\in \ZS_n^0$ where $t_i=s_i$ whenever $s_i\neq 0$, $\ZS_n^0$ is in $E(\ZS_n)$. On the other hand the element $t\in \ZS_n^0$ is eliminated by $t' \in \ZS_n^0$ if and only if $t=t'$. Therefore $\ZS_n^0$ is minimal.
\end{proof}

\begin{example}\label{ex:color}
Define $T_{\text{red}}=\{0,1\}$, $T_{\text{green}}=\{0,1\}$, $T_{\text{blue}}=\{0,1\}$, and $$T_{\text{color}}=\left\{
\begin{array}{c}
\underset{\text{black}}{\underbrace{(1,0,0,0)}},
\underset{\text{red}}{\underbrace{(0,1,0,0)}},
\underset{\text{green}}{\underbrace{(0,0,1,0)}},
\underset{\text{blue}}{\underbrace{(0,0,0,1)}},
\underset{\text{purple}}{\underbrace{(0,1/2,0,1/2)}}, \\
\underset{\text{yellow}}{\underbrace{(0,0,1/2,1/2)}},
\underset{\text{aqua}}{\underbrace{(0,1/2,1/2,0)}},
\underset{\text{white}}{\underbrace{(0,1/3,1/3,1/3)}}
\end{array}
\right\}$$
Let $$\phi :T_{\text{red}}\times T_{\text{green}}\times T_{\text{blue}}\rightarrow T_{\text{color}}$$ be the multi-valued logic gate which sends the input $(r,g,b)$ to a point which corresponds to the color obtained by mixing the colors whose truth value is $1$. For example,
$\phi(0,0,0)=(1,0,0,0)$, $\phi(0,1,0)=(0,0,1,0)$ and $\phi(0,1,1)=(0,0, 1/2,1/2)$. The induced multi-valued logic gate
$$\overline{\phi}:\Delta^1 \times \Delta^1 \times \Delta^1 \rightarrow \Delta ^3$$
is given by
$$\overline{\phi}((r_0,r_1),(g_0,g_1),(b_0,b_1))=(f_1,f_2,f_3,f_4)$$
where
 $$\begin{array}{c}
f_1=r_0*g_0*b_0\\
f_2=r_1*g_0*b_0+1/2*r_1*g_0*b_1+1/2*r_1*g_1*b_0+1/3*r_1*g_1*b_1\\
f_3=r_0*g_1*b_0+1/2*r_0*g_1*b_1+1/2*r_1*g_1*b_0+1/3*r_1*g_1*b_1\\
f_4=r_0*g_0*b_1+1/2*r_0*g_1*b_1+1/2*r_1*g_0*b_1+1/3*r_1*g_1*b_1.\\
\end{array}$$
Therefore we have
\begin{eqnarray*}\frac{\partial }{\partial r_1}(f_2-f_3-f_4)&=&g_0*b_0+g_0*b_1+g_1*b_0+2/3*g_1*b_1>0\\
\frac{\partial }{\partial g_1}(f_2-f_3-f_4)&=&-r_0*b_0-r_1*b_0-1/3*r_1*b_1<0\\
\frac{\partial }{\partial b_1}(f_2-f_3-f_4)&=&-r_0*g_0-r_1*g_0-1/3*r_1*g_1<0\end{eqnarray*}
and hence
$$\mathrm{Sign}_C(f_2-f_3-f_4)=(1,-1,-1)$$
where $C=C(\overline{\phi},(0,0,0))$. Similarly one can show that
\begin{eqnarray*}\mathrm{Sign}_C(f_2+f_3+f_4)&=&(1,1,1)\\
\mathrm{Sign}_C(f_2-f_3+f_4)&=&(1,-1,1)\\
\mathrm{Sign}_C(f_2+f_3-f_4)&=&(1,1,-1).\end{eqnarray*}
Since
$$\ZS_3^0=\{(1,1,1),(1,-1,-1),(1,-1,1),(1,1,-1)\},$$
we can conclude that $\overline{\phi}$ is a reversible gate by repeated application of the above lemma and the main result of this section.
\end{example}

\section{Computing Continuous Sensitivity}

We could use experimental data about a multi-valued logic gate to obtain an upper bound on the continuous sensitivity of the gate due to the following simple lemma.
\begin{lemma} If $X$ and $Z$ are subsets of $\ZS _N$ and $Z\cap \ZE (X)=\emptyset $ then
$$|\ZE (Z)|\leq |\ZE (S_N-\ZE (X))|$$
\end{lemma}
In the above lemma consider $X=\mathrm{Sens}_C(f)$ and $Z$ as signs eliminated by experimental data.
As an application of this lemma we can show that sensitivity (see Section 2 in \cite{Ambainisetal}) of a boolean function is an upper bound for its continuous sensitivity.

\begin{definition}
  Let $\phi :\{0,1\}^N\rightarrow\{0,1\}$ be a boolean function and $z$ be an element in $\{0,1\}^N$. Then the sensitivity of $\phi $ at the input $z$ is defined as follows:
  $$s(\phi ,z)=\text{ number of indices }i\text{ such that }\phi(z)\neq \phi(z^i)$$
  where $z^i$ denotes the element in $\{0,1\}^N$ obtained by changing the $i^{th}$ coordinate of $z$. The sensitivity of $\phi $ is defined as follows:
  $$s(\phi )=\max \left\{\,s(\phi ,z)\,|\,z\in \{0,1\}^N\,\right\}.$$
\end{definition}

\begin{theorem}
  Let $\phi :\{0,1\}^N\rightarrow\{0,1\}$ be a boolean function. Then
  $$ cs(\overline{\phi})\leq s(\phi ).$$
\end{theorem}
\begin{proof}
  Let  $z=(z_1,z_2,\dots,z_N)$ be an element in $\{0,1\}^N$. Define $D_z=\{\,i\,|\, \phi(z)= \phi(z^i)\, \}$. Then we have
  $$s(\phi ,z)=N-|D_z|.$$
Notice that by Mean Value Theorem for every $i$ in $D_z$ there exists $c_{i,z}$ in $C(\overline{\phi },z)$ such that $$\frac{\partial \overline{\phi }}{\partial x_i}(z_1,z_2,\dots,z_{i-1},c_{i,z},z_{i+1},z_{i+2},\dots ,z_N)=0,$$
hence $e_i$ is not $\mathrm{Sens}_{C(\overline{\phi },z)}(f(\overline{\phi }))$ for all $i$ in $D_z$. Therefore, we  have
$$\left|\ZE(\ZS _{N}-\mathrm{Sens}_{C(\overline{\phi },z)}(f(\overline{\phi })))\right|
\geq
|\ZE(\{\,e_i\,|\,i\in D_z\,\})|=\frac{3^N-3^{N-|D_z|}}{2}$$
and hence
$$cs(\overline{\phi},z)=
\log _{3}\left( 3^{N}-2\left| \ZE(\ZS _{N}-\mathrm{Sens}_{C(\overline{\phi },z)}(f(\overline{\phi })))\right|\right)
\leq N-|D_z|= s (\phi ,z)$$
Therefore the result follows.
 \end{proof}

The above result shows that it is important to know how to count eliminated signs. For the rest of this section we will discuss methods for counting eliminated signs. Notice that the number of elements eliminated by each element of $\ZS_2$ is the same. This is not true in general. For example for $x=(0,0,1)$ and $y=(1,1,1)$ in $\ZS_3$, we have
{\small
\begin{eqnarray*} \ZE(\{x\})&=&\ZS_3 - \{(0,1,0),(1,0,0),(1,1,0),(1,-1,0)\},\\
\ZE(\{y\})&=&\ZS_3 - \{(0,1,-1),(1,0,-1),(1,1,-1),(1,-1,0),(1,-1,1),(1,-1,-1)\}
\end{eqnarray*}}and hence $|\ZE(\{x\})| \neq |\ZE(\{y\})|$.
In general the size of $\ZE(\{x\})$ depends only on the number of zeros of $x$ and is given as follows.
\begin{lemma}\label{lem:counting-single} For $x \in \ZS_n$, $|\ZE(\{x\})|=3^{z(x)}(2^{n-z(x)}-1)$ where $z(x)$ is the number of zeros of $x$.
\end{lemma}

\begin{proof} Let $x=(x_1,\cdots, x_n) \in \ZS_n$.  We first consider the case  $z(x)=0$. In this case $x$ eliminates $s=(1,s_2,\cdots, s_n)$ if and only if $s_i=0$ or $x_i$ for $i \geq 2$. If the first nonzero term of $s$ is $k$-th one then $x$ eliminates $s$ either $s_i \in \{0,x_i\}$ or $s_i \in \{0,-x_i\}$. Therefore when $z(x)=0$, we have
$$|\ZE(\{x\})|=2^{n-1}+\sum_{k=2}^n2^{n-k}=2^n-1.$$

Now suppose that $z(x) \neq 0$. We prove this case by induction on $n$. The case $n=2$ follows from Example \ref{ex:base}. Let $x'=(x_1,\cdots,\widehat{x_{j}},\cdots, x_n) \in \ZS_{n-1}$ where $j=\mathrm{min}\{i| t_i=0\}$. When $j=1$
$$\ZE(\{x\})= \{(\varepsilon,s_1,\cdots,s_{n-1}), \ (1,-s_1,\cdots, -s_{n-1})| \ \varepsilon \in\{0,1\},  \ (s_1,\cdots,s_{n-1})\in \ZE(x')\},$$
and otherwise we have
$$\ZE(\{x\})= \{(s_1, \cdots, \varepsilon, \cdots,s_{n-1})| \ \varepsilon \in\{-1,0,1\},  \ (s_1,\cdots,s_{n-1})\in \ZE(x')\}.$$ Therefore the number of elements of $\ZE(x)$ is three times the number of elements of $\ZE(x')$ and hence the result follows by induction.
\end{proof}

Now we generalize the above lemma to the intersections of eliminated sets for $m$-many elements in $\ZS_n$.  Let $z_X$ be the number of zero columns of $M_X$.

\begin{theorem}\label{thm:counting-intersctions} For $X=\{X^1,\dots,X^m \}\subseteq \ZS_n$, we have
$$\big|\bigcap_{i=1}^m \ZE(X^i)\big|=3^{z_{X}}\Big(-(-2)^{m-1}+\sum_{\alpha \in \ZS_m}(-1)^{z(\alpha)}2^{z(\alpha)+|\mathrm{BSp}(X,\alpha)|}\Big)$$where for any $\alpha \in \ZS_m$,
$\mathrm{BSp}(X,\alpha)$ is the set of columns of $M_X$ of the form $\pm \sum_{i=1}^ma_i\alpha_ie_i$ with $a_i \in \{0,1\}^m - \{(0,\dots,0)\}$.
\end{theorem}

We prove this theorem using the following lemma.

\begin{lemma}Let $X=\{X^1,\dots,X^m \}\subseteq \ZS_n$ where $X^j=(X^j_1,\dots,X^j_n)$ be such that $M_X$ has no zero column. Then the number of elements of the form $(1,s_2,\cdots,s_{n})$ in $\bigcap_{i=1}^m \ZE(X^i)$ is
$$\sum_{\alpha \in A} (-1)^{z(\alpha)}2^{z(\alpha)+|\mathrm{BSp}(X,\alpha)|-1}$$ where $$A=\{(\alpha_1,\dots,\alpha_m)\in \ZS_m| \ \alpha_j=1 \ \mathrm{whenever} \ X^j_1\neq0\}.$$
\end{lemma}
\begin{proof} Let $c^i=(X^1_i,\dots, X^m_i)^T$ be the $i$-th column of $M_X$ for $1\leq i \leq n$. By reordering elements of $X$, we can assume that $c^1=e_1^T+\dots+e_r^T$ for some $1\leq r\leq m$. In this case, we have $A=\{(1,\dots,1,\alpha_{r+1},\dots, \alpha_m)| \ \alpha_{r+i}\in\{-1,0,1\}\}$. Note that the $i$-th column $c^i$ is not in $\underset{\alpha \in A}{\bigcup}\mathrm{Bsp}(X,\alpha)$ if and only if $c^i_j=1$ and $c^i_{j'}=-1$ for some $1\leq j,j'\leq r$. Moreover if $s=(1,s_2,\dots,s_{n})\in  \overset{m}{\underset{i=1}{\bigcap}} \ZE(X^i)$ then $s_i=0$ for all $i$ for which $c^i \notin \underset{\alpha \in A}{\bigcup}\mathrm{Bsp}(X,\alpha)$. So without loss of generality we can assume that $c^i \in \mathrm{Bsp}(X,\alpha)$ for some $\alpha=(1,\dots,1,\alpha_{r+1},\dots,\alpha_m)\in \ZS_m$ for all $i$, that is, $c^i=\pm \sum a_j^i\alpha_je_j$ where $a_j^i\in \{0,1\}$ for $1\leq i \leq m$.

For each $\alpha=(1,\dots,1,\alpha_{r+1},\dots,\alpha_m)\in \ZS_m$ with $\alpha_i\neq 0$ for any $i$, let $A_{\alpha}$ be the set of all $s=(1,s_2,\dots,s_n)$ where $s_i \in \{0,1\}$ if $c^i=\sum a_j^i\alpha_je_j$ and $s_i\in\{0,-1\}$ if $c^i=-\sum a_j^i\alpha_je_j$. Here, $|A_{\alpha}|=2^{|\mathrm{BSp}(X',\alpha)|}$. An element $s=(1,s_2,\dots,s_n) \in \ZS_n$ lies in the intersection of $A_{\alpha}$ and $ \overset{m}{\underset{i=1}{\bigcap}} \ZE(X^i)$ if and only if $s$ and $X^i$ has common non-zero elements for each $r+1 \leq i \leq m$. Clearly, $s\in A_{\alpha}$ does not satisfy this property if there exists $r+1\leq k\leq m$ such that $a_k^i\neq0 $ implies $s^i=0$. To eliminate these terms, we first need to remove the ones with $s_i=0$ for all $a_k^i\neq 0$ for all $r+1\leq k \leq m$. For each $k$, there are $2^{|\mathrm{BSp}(X',\alpha^k)|}$ many such $s$ in $A_{\alpha}$ where $\alpha_j^k=\alpha_j$ if $j\neq k$ and $\alpha_k^k=0$. Then we need to add the ones with $s_i=0$ for all $a^i_k \neq 0$ or $a^i_{k'}\neq 0$ for $r+1 \leq k,k' \leq m$ since we remove them twice. There are $2^{|\mathrm{BSp}(X',\alpha^{k,k'})|}$ many such $s$ in $A_{\alpha}$ where $\alpha_j^{k,k'}=\alpha_j$ if $j \neq k$ or $k'$ and $\alpha_k^{k,k'}=\alpha_{k'}^{k,k'}=0$. Then we need to remove the ones corresponding to triples since we add them twice. By continuing in this way, we obtain that
$$\big|A_{\alpha}\bigcap \Big(\bigcap_{i=1}^m \ZE(X^i)\Big)\big|= \sum_{\beta \in S_{\alpha}} (-1)^{z(\beta)}2^{|\mathrm{BSp}(X',\beta)|}$$where $S_{\alpha}=\{\beta=(1,\dots,1,\beta_{r+1},\dots,\beta_m)| \ \beta_j=\alpha_j \ \mathrm{or} \ 0, \ r+1 \leq j \leq m \}$.

Let $\alpha=(1,\dots,1,\alpha_{r+1,\dots,m})$ and $\gamma=(1, \dots, 1, \gamma_{r+1},\dots,\gamma_m)$ be distinct elements of $\{\pm1\}^m$, i.e, there exist $k$ such that $\alpha_k=-\gamma_k$. If $s \in A_{\alpha} \bigcap A_{\gamma}$ then whenever $a^i_k\neq 0$. So we have $A_{\alpha}\bigcap A_{\gamma} \bigcap \Big( \overset{m}{\underset{i=1}{\bigcap}} \ZE(X^i)\Big) = \emptyset$. Clearly, if $s=(1,s_2, \dots, s_n)$ is in the intersection of $\ZE(X^i)$'s then $s \in A_{\alpha}$. Therefore we have
$$\big|\bigcap_{i=1}^m \ZE(X^i)\big|=\sum_{\alpha=(1,\dots,1,\alpha_{r+1},\dots, \alpha_m)\in \{\pm 1\}^m} \ \sum_{\beta \in S_{\alpha}}(-1)^{z(\beta)}2^{|\mathrm{BSp}(X',\beta)|}.$$ Since $\beta$ is an element of $S_{\alpha}$ for $2^{z(\beta)}$-many distinct $\alpha$'s, we have
$$\big|\bigcap_{i=1}^m \ZE(X^i)\big|=\sum_{\beta=(1,\dots,1,\beta_{r+1},\dots, \beta_m)\in \ZS_m}(-1)^{z(\beta)}2^{z(\beta)+|\mathrm{BSp}(X',\beta)|}.$$ Since $|\mathrm{BSp}(X',\beta)|=|\mathrm{BSp}(X,\beta)|-1$, the result follows.
\end{proof}
\begin{proof}[Proof of Theorem \ref{thm:counting-intersctions}] Let $X^i=(X^i_1,\dots, X^i_n)$ for $1\leq i\leq m$. We proceed by induction on $n$. The case $n=2$ follows from Example \ref{ex:base}. Note that in this case the size of the intersection of eliminated set of two different elements of $\ZS_2$ is $2$, three different elements is $1$ and the intersection of eliminated sets of all is 0. For $n>2$, we first consider the case where $M_X$ has a zero column. Let $\widetilde{X}=\{\widetilde{X}^1,\dots,\widetilde{X}^m\}$ where $\widetilde{X}^i=(X^i_1,\dots,X^i_{k-1},X^i_{k+1}\dots,X^i_{n})$ if $k$-the column of $M_X$ is zero. Then
$$\big|\bigcap_{i=1}^m \ZE(\widetilde{X}^i)\big|=3^{z_{\widetilde{X}}}\Big(-(-2)^{m-1}+\sum_{\alpha\in \ZS_m}(-1)^{z(\alpha)}2^{z(\alpha)+|\mathrm{BSp}(\widetilde{X},\alpha)|}\Big)$$by induction hypothesis. Note that if $k \neq 1$ then $s=(s_1,\dots,s_n)$ is in  $ \overset{m}{\underset{i=1}{\bigcap}} \ZE(X)$ if and only if $(s_1,\dots,s_{k-1},s_{k+1},\dots,s_{n}) \in \overset{m}{\underset{i=1}{\bigcap}} \ZE(\widetilde{X}^i)\},$ and $s_{k} \in \{0,1,-1\}$. If $k=1$ then the elements in $ \overset{m}{\underset{i=1}{\bigcap}} \ZE(X)$ are of the form $(\varepsilon,s_2,\dots,s_n)$ where $\varepsilon\in \{0,1\}$ and $(s_2,\dots,s_n)\in\bigcap_{i=1}^m \ZE(\widetilde{X})$ or $\varepsilon=-1$ and $-(s_2,\dots,s_n)\in \overset{m}{\underset{i=1}{\bigcap}} \ZE(\widetilde{X})$. So the result follows for this case since $\mathrm{BSp}(X,\alpha)=\mathrm{BSp}(\widetilde{X}, \alpha)$.

Now, suppose that $M_X$ has no zero column. By reordering elements of $X$, we can assume that first column of $M_X$ is of the form $e_1+\dots+e_r$ for some $1\leq r\leq m$. Let $i_1,\dots, i_k$ be the set of all $i$'s for which $(X^j_2,\dots,X^j_n)\notin \ZS_{n-1}$. Clearly, $1\leq i_1,\cdots, i_k \leq r$. Then $\overline{X}=\{\overline{X^1},\dots, \overline{X^m}\}\subset \ZS_{n-1}$ where $\overline{X^j}=(X^j_2,\dots,X^j_n)$ if $j\neq i_t$ for some $1\leq t \leq k$ and $\overline{X^j}=-(X^j_2,\dots,X^j_n)$, otherwise. Then $s=(0,s_2,\dots,s_n)\in \ZS_n$ is in $\overset{m}{\underset{i=1}{\bigcap}} \ZE(X^i)$ if and only if $(s_2,\dots,s_{n-1})$ is in $\overset{m}{\underset{i=1}{\bigcap}}\overline{X^i}$. Therefore the size of the elements of this form in $\overset{m}{\underset{i=1}{\bigcap}} \ZE(X^i)$ is
$$-(-2)^{m-1}+\sum_{\alpha \in \ZS_m}(-1)^{z(\alpha)}2^{z(\alpha)+|\mathrm{BSp}(\overline{X},\alpha)|}$$ by induction. Note that for $\alpha=(\alpha_1,\cdots,\alpha_m)$ and $\alpha'=(\alpha'_1,\cdots,\alpha_m')$ in $\ZS_m$ with $\alpha'_j=-\alpha_j$ for $j=i_1,\dots,i_k$ and $\alpha'_j=\alpha_j$ otherwise, we have
$$|\mathrm{BSp}(\overline{X},\alpha)|=|\mathrm{BSp}(X,\alpha')| \ \mathrm{and} \ z(\alpha)=z(\alpha').$$ Moreover for $\alpha=(1,\dots,1,\alpha_{r+1},\dots, \alpha_m) \in \ZS_m$, $|\mathrm{BSp}(\overline{X},\alpha)|=|\mathrm{BSp}(X,\alpha)|-1$.

On the other hand, there are
$$\sum_{\alpha=(1,\dots,1,\alpha_{r+1},\dots,\alpha_m) \in \ZS_m}(-1)^{z(\alpha)}2^{z(\alpha)+|\mathrm{BSp}(X,\alpha)|-1}$$
many elements of the form $s=(1,s_2,\dots,s_n)\in \ZS_n$ is in $\bigcap_{i=1}^m \ZE(X^i)$ by previous lemma. Therefore the total number of elements in the intersection of the $\ZE(X^i)$'s is
$$-(-2)^{m-1}+\sum_{\alpha \in \ZS_m}(-1)^{z(\alpha)}2^{z(\alpha)+|\mathrm{BSp}(X,\alpha)|}$$ in this case.
\end{proof}

Now one can use the inclusion-exclusion principle, to find the number of elements eliminated by an arbitrary subset $X=\{X^1,\dots,X^t\}$ of $\ZS_n$.

\begin{prop}\label{prop:counting-set} The size of the set of elements eliminated by $X=\{X^1,\cdots,X^m\}$ is
$$\sum_{k=1}^m \sum_{1\leq i_1 < \dots < i_k \leq m}3^{z_{X(i_1,\dots,i_k)}}\Big(-2^{k-1}+ \sum_{\alpha \in S_k} (-1)^{z(\alpha)+k+1}2^{z(\alpha)+|\mathrm{BSp}(X(i_1,\dots,i_k), \alpha)|}\Big)$$ where $X(i_1,\dots,i_k)=\{X^{i_1},\dots,X^{i_k}\}.$
\end{prop}

To make the calculations easier, we introduce two column operations on the set of matrices of the form $M_X$. First one is the action of the symmetric group $S_n$ of order $n$. We define the action of $\sigma \in S_n$ on $\ZS_n$ by
$$ \sigma(x)=\left\{
  \begin{array}{ll}
    (x_{\sigma(1)},\cdots,x_{\sigma(n)}), & \hbox{if $(x_{\sigma(1)},\cdots,x_{\sigma(n)}) \in \ZS_n$;} \\
    (-x_{\sigma(1)},\cdots,-x_{\sigma(n)}), & \hbox{otherwise.}
  \end{array}
\right.$$ Here $(x_{\sigma(1)},\cdots,x_{\sigma(n)}) \in \ZS_n$ means that the first non-zero term of $(x_{\sigma(1)},\cdots,x_{\sigma(n)})$ is $1$. This induces an action of $\sigma$ on $\{M_X| \ X \subseteq \ZS_n\}$ by sending $M_X$ to $M_{\sigma(X)}$ where $\sigma(X)=\{\sigma(X^1),\dots,\sigma(X^n)\}$. The second operation is multiplying a column of $M_{X}$ by $-1$. We also need to be careful here. For $x=(x_1,\dots,x_n) \in \ZS_n$, let
$$x_{j}^{-}=\left\{
               \begin{array}{ll}
                 (x_1,\dots,x_{j-1},-x_j,x_{j+1},\dots,x_n), & \hbox{if $(x_1,\dots,x_{j-1},-x_j,x_{j+1},\dots,x_n) \in \ZS_n$;} \\
                 -(x_1,\dots,x_{j-1},-x_j,x_{j+1},\dots,x_n), & \hbox{otherwise.}
               \end{array}
             \right.
$$Here what we mean by the matrix obtained by multiplying $j$-th column of $M_{X}$ by $-1$ is the matrix $M_{X[j]}$ where $X[j]=\{(X^1)_{j}^{-},\dots,(X^m)_{j}^{-}\}$. Clearly these operations preserves the cardinality of arbitrary intersections of eliminated sets.

\begin{example} Let $X=\{x, y \}\subset \ZS_n$. By applying column operations, we can assume that $M_X$ has 5 types of columns namely
$$\left(
    \begin{array}{c}
      1 \\
      1 \\
    \end{array}
  \right), \ \left(
    \begin{array}{c}
      1 \\
      -1 \\
    \end{array}
  \right), \ \left(
    \begin{array}{c}
      1 \\
      0 \\
    \end{array}
  \right), \ \left(
    \begin{array}{c}
      0 \\
      1 \\
    \end{array}
  \right), \ \mathrm{and} \ \left(
    \begin{array}{c}
      0 \\
      0 \\
    \end{array}
  \right).
$$ Suppose that $M_X$ has $a_1$-many columns of first type, $a_2$-many columns of second type, $b_1$-many columns of third type, $b_2$-many columns of forth type and $c$-many columns of last type. This means that $x$ has $(b_2+c)$-many zeros and $y$ has $(b_1+c)$-many zeros. By above Proposition, we have
\begin{eqnarray*}|\ZE(X)|=
3^c\Big(3^{b_1}2^{a_1+a_2+b_2}+3^{b_2}2^{a_1+a_2+b_1}-2^{b_1+b_2}(2^{a_1}+2^{a_2})-3^{b_1}-3^{b_2}+2^{b_1+1}+2^{b_2+1}-2\Big)\end{eqnarray*}
\end{example}

\bibliographystyle{amsplain}
\bibliography{CSRBib}

\end{document}